\documentclass[11pt,a4paper]{article}
\usepackage{amsmath,amssymb,amsthm}
\usepackage{epsfig,epstopdf}
\usepackage{latexsym}
\usepackage{graphics,graphicx}
\usepackage[usenames]{color}
\usepackage[german,english]{babel}
\usepackage{url}
\usepackage{hyperref}
\usepackage{lscape}
\usepackage{booktabs}
\usepackage{tabularx}
\usepackage{multicol,placeins}
\usepackage{cite}
\usepackage[super,sort&compress,comma]{}
\usepackage[german,english]{babel}
\usepackage{tikz}
\usetikzlibrary{cd}
\usepackage{fullpage}
\usepackage{setspace}
\date{}

\newtheorem{df}{Definition}[section]
\newtheorem{pro}{Proposition}[section]
\newtheorem{thm}{Theorem}[section]
\newtheorem{cor}{Corollary}[section]
\newtheorem{ex}{Example}[section]
\newtheorem{lem}{Lemma}[section]
\newtheorem{rem}{Remark}[section]

\makeatletter

\renewcommand\section{\@startsection {section}{1}{\z@}
{-30pt \@plus -1ex \@minus -.2ex} {2.3ex \@plus.2ex}
{\normalfont\normalsize\bfseries}}

\renewcommand\subsection{\@startsection{subsection}{2}{\z@}
{-3.25ex\@plus -1ex \@minus -.2ex} {1.5ex \@plus .2ex}
{\normalfont\normalsize\bfseries}}

\renewcommand{\@seccntformat}[1]{\csname the#1\endcsname. }

\makeatother

\begin{document}
\title{Armendariz ring with weakly semicommutativity}
\author{\small{Sushma Singh and Om Prakash} \\ \small{Department of Mathematics} \\ \small{Indian Institute of Technology Patna, Bihta} \\ \small{Patna, INDIA- 801 106} \\  \small{sushmasingh@iitp.ac.in \& \; om@iitp.ac.in }}

\maketitle
\begin{abstract}
In this article, we introduce the weak ideal-Armendariz ring which combines Armendariz ring and weakly semicommutative properties of rings. In fact, it is a generalisation of an ideal-Armendariz ring. We investigate some properties of weak ideal Armendariz rings and prove that $R$ is a weak ideal-Armendariz ring if and only if $R[x]$ is weak ideal-Armendariz ring. Also, we generalise weak ideal-Armendariz as strongly nil-IFP and a number of properties are discussed which distinguishes it from other existing structures. We prove that if $I$ is a semicommutative ideal of a ring $R$ and $\frac{R}{I}$ is a strongly nil-IFP, then $R$ is strongly nil-IFP. Moreover, if $R$ is 2-primal, then $R[x]/<x^{n}>$ is a strongly nil-IFP.
\end{abstract}


\section{INTRODUCTION}
Throughout this article, $R$ denotes an associative ring with identity, otherwise it is mentioned and $R[x]$ is a polynomial ring over $R$ with an indeterminate $x$. For any polynomial $f(x)$, $C_{f(x)}$ denotes the set of all coefficients of $f(x)$. $M_{n}(R)$ and $U_{n}(R)$ denote the $n\times n$ full matrix ring and upper triangular matrix ring over $R$ respectively. $D_{n}(R)$ is the ring of $n\times n$ upper triangular matrices over $R$ whose diagonal entries are equal and $V_{n}(R)$ is ring of all matrices $(a_{ij})$ in $D_{n}(R)$ such that $a_{pq} = a_{(p+1)(q+1)}$ for $p = 0, 1, 2, \ldots, n-1$, $q = 0, 1, 2, \ldots, n-1$.  We use $e_{ij}$ to denote the matrix with $(i, j)$th-entry $1$ and elsewhere $0$. $\mathbb{Z}_{n}$ is the ring of residue classes modulo a positive integer $n$ and $GF(p^{n})$ denotes the Galois field of order $p^{n}$. Here, $N_{*}(R),~ N(R)$ and $J(R)$ represent the prime radical (lower nil radical), set of all nilpotent elements and Jacobson radical of the ring $R$ respectively. $R^{+}$ represents the additive abelian group $(R, +)$ and cardinality of a given set $S$ is denoted by $|S|$.\\
Recall that a ring $R$ is reduced if it has no non-zero nilpotent element. It was Armendariz [\cite{E}, Lemma 1] who initially observed that a reduced ring always satisfies the below condition: \\
\emph{If $f(x)$ and $g(x)\in R[x]$ satisfies $f(x)g(x) = 0$, then $ab = 0$ for each $a \in C_{f(x)}$ and $b \in C_{g(x)}$}.\\
In 1997, Rege and Chhawchharia \cite{M} had coined the term Armendariz for those rings which are satisfying above condition. Clearly, subring of an Armendariz ring is an Armendariz.
 A ring $R$ is said to be an abelian ring if every idempotent element is central. It is known that Armendariz ring is an abelian ring\cite{C}.\\
In 2006, Liu et al. generalised the Armendariz ring which is called weak Armendariz\cite{Z} and defined as if two polynomials $f(x)$ and $g(x)\in R[x]$ such that $f(x)g(x) = 0$ implies $ab \in N(R)$ for each $a \in C_{f(x)}$ and $b \in C_{g(x)}$ .
  A ring $R$ is reversible if, for any  $a, b \in R, ab =0$ implies $ba = 0$ \cite{P}. Later, by using the concept of reversible, D. W. Jung et. al \cite{DD} introduced quasi-reversible-over-prime-radical $(QRPR)$. A ring $R$ is said to be $(QRPR)$ if $ab = 0$ implies $ba \in N_{*}(R)$ for $a, b \in R$.\\
A right (left) ideal $I$ of a ring $R$ is said to have the insertion-of-factor-property (IFP) if $ab \in I$ implies $aRb \subseteq I$ for $a, b \in R$. A ring $R$ is said to be an IFP if the zero ideal of $R$ has the IFP Property \cite{H}. IFP rings are also known as semicommutative rings \cite{LL}. In 1992, Birkenmeier et. al \cite{G},  are called a ring $R$ is 2-primal if and only if  $N_{*}(R) = N(R)$. It can be easily seen that semicommutative rings are $2$- primal. A ring $R$ is called an $NI$ ring if $N^{*}(R) = N(R)$ \cite{GG}.\\
In 2012, Kwak \cite{TT}, studied Armendariz ring with IFP and introduced ideal-Armendariz ring. He proved that a ring $R$ is Armendariz and IFP if and only if $f(x)g(x) = 0$ implies $aRb = 0$, for each $a \in C_{f(x)}$ and $b \in C_{g(x)}$.\\
Later, in 2007, Liang \cite{L} introduced the weakly semicommutative ring which satisfy for any $a, b \in R$, $ab = 0$ implies $arb\in N(R)$ for each $r \in R$. By definition, it is clear that any semicommutative ring is weakly semicommutative ring but converse is not true. Also, for a reduced ring $R$, $D_{n}(R)$ is weakly semicommutative ring for $n\geq 4$ but it is not semicommutative. Therefore, $D_{n}(R)$ is a weakly semicommutative but by Example $3$ of \cite{N}, for $n \geq 4$, it is not an Armendariz ring. Moreover, by Example(4.8) of \cite{RR}, if $K$ is a field and $R = K[a, b]/<a^{2}>$, then $R$ is an Armendariz ring. In this ring,  $(ba)a = 0$ but $(ba)b(a)$ is not nilpotent. Therefore, $R$ is not a weakly semicommutative ring. Thus, Armendariz ring and weakly semicommutative ring do not imply each other.\\

In light of above study, we introduced the concept of weak ideal-Armendariz ring and Strongly nil-IFP in Section 2 and Section 3 respectively. The below flow chart will also give an idea about relation of these two notions with other existing structures:\\

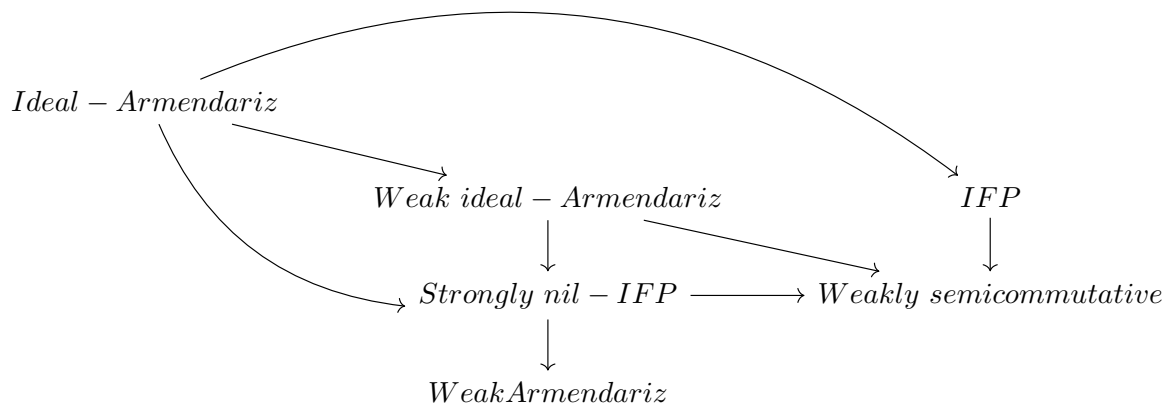
\begin{tikzcd}
Ideal-Armendariz
\arrow[drr, bend left,]
\arrow[ddr, bend right,]
\arrow[dr,  description] & & \\
& Weak ~ideal-Armendariz  \arrow[d, ] \arrow[dr, ]
& IFP \arrow[d, ] \\
& Strongly ~nil-IFP \arrow[r, ] \arrow[d, ]
& Weakly~semicommutative
&  \\
& Weak Armendariz
\end{tikzcd}

\section{WEAK IDEAL ARMENDARIZ RINGS}
\begin{df}{Weak Ideal Armendariz Rings} A ring $R$ is said to be a weak ideal-Armendariz ring if $R$ is an Armendariz ring and weakly semicommutative.
\end{df}
It is clear that every ideal-Armendariz ring is a weak ideal-Armendariz but converse is not true. Towards this, we have the following example:\\

\begin{ex} Let $F$ be a field and $A = F[a, b, c]$ be a free algebra of polynomials with constant terms zero in noncommuting indeterminates $a, b, c$ over $F$. Then $A$ is a ring without identity. Consider an ideal $I$ of $F+A$ generated by $cc, ac, crc$ for all $r \in A$. Let $R = (F+A)/I$. From Example 14 of \cite{C}, $R$ is an Armendariz but it is not a semicommutative ring because $ac \in I$ but $abc \notin I$. We denote $a+I$ by $a$. In order to show $R$ is a weakly semicommutative ring, we have the following cases: \\
\textbf{Case (1)} $ac \in I$, \textbf{Case (2)} $cc \in I$,\textbf{Case (3)} $crc \in I$.\\
\textbf{Case (1)} If $ac \in I$ for $a, c \in A$, then $(asc)^{2} \in I$ for each $s \in A$. Consider, $atc = a(\alpha+t^{'})c$, where $t = \alpha+t^{'} \in F+A$. Then $(atc)^{2} \in I$. Since, $r_{1}acr_{2} \in I$, therefore  $(r_{1}arcr_{2})^{2}\in I$ for any $r_{1}, r_{2}, r \in A$. Similarly, we can prove that $(r_{1}atcr_{2})^{2}\in I$ for any $r_{1}, r_{2},  \in A$, and  $t\in F+A$. Again, we have $(r_{1}acr_{2})\in I$ for any $r_{1}, r_{2} \in F+A$ and this implies $(r_{1}arcr_{2})^{2} \in I$ for any $r\in A$ as well as $(r_{1}atcr_{2})^{2} \in I$ for $t \in F+A$. \\
 Similarly, rest cases can be easily checked.
\end{ex}

\begin{rem}
\begin{itemize}
\item[$(1)$] Every reduced ring is a weak ideal-Armendariz ring.
\item[$(2)$] Subring of a weak ideal-Armendariz ring is a weak ideal-Armendariz ring.
\item[$(3)$] If $R$ is a reduced ring, then $D_{3}(R)$ is a weak ideal-Armendariz ring.
\item[$(4)$]  For any ring $R$ and $n\geq 2, $ $U_{n}(R)$ and $M_{n}(R)$ are not Armendariz by \cite{N} and \cite{M} respectively. Hence, they are not weak ideal-Armendariz rings.
\item[$(5)$] By Example $3$ of \cite{N}, for any ring $A$, $D_{n}(A)$ for $n\geq 4$ is not an Armendariz ring. Hence, it is not a weak ideal Armendariz ring. But, when $A$ is a reduced ring, then $D_{n}(A)$ is a weakly semicommutative ring.
\item[$(6)$] It is nevertheless to say that polynomial ring over weakly semicommutative ring need not be weakly semicommutative. Therefore, there is a huge scope to extend and study the polynomial rings over weakly semicommutative rings with the Armendariz condition.
    \end{itemize}
\end{rem}
\begin{pro} A ring $R$ is weak ideal-Armendariz ring if and only if $R[x]$ is weak ideal-Armendariz ring.
\end{pro}
\begin{proof} Since, $R$ is Armendariz ring if and only if $R[x]$ is Armendariz ring\cite{D}. In order to prove the result, we show that $R$ is weakly semicommutative ring if and only if $R[x]$ is weakly semicommutative. For this, let  $f(x)$ and $g(x) \in R[x]$ such that $f(x)g(x) = 0$. Since $R$ is Armendariz, therefore $ab = 0$ for each $a \in C_{f(x)}$ and $b \in C_{g(x)}$. Since $R$ is weakly semicommutative ring, therefore $arb \in N(R)$ for each $r \in R$. Hence, $f(x) R[x] g(x) \subseteq N(R)[x]$. By Lemma (2.6) of \cite{RR}, we have $N(R)[x]\subseteq N(R[x])$, hence $f(x)R[x]g(x) \subseteq N(R[x])$. Thus, $R[x]$ is weak ideal-Armendariz ring.
Also, $R$ being subring of $R[x]$, converse is true.
\end{proof}

\begin{pro} Let $N$ be a nil ideal of a ring $R$. If $|N| = p^{2}$, where $p$ is a prime, then $N$ is a commutative Armendariz ring without identity such that $N^{3} = 0$.
\end{pro}
\begin{proof} If $|N| = p^{2}$, then $N$ is nilpotent by Hungerford [Proposition 2.13 of \cite{T}], since $J(N) = N \cap J(R)$.\\
By Corollary $(1.3.11)$ of \cite{R}, if $R$ is nilpotent ring of order $p^{n}$, then exp$(R)\leq n+1$. Hence $N^{3} = 0$.

Suppose $N^{+}$ is cyclic. To prove $N$ is Armendariz, let $f(x), g(x) \in N[x]$ such that $f(x)g(x) = 0$. Then $f(x), g(x)$ can be written as $f(x) = af_{1}(x)$ and $g(x) = ag_{1}(x)$, where $f_{1}(x), g_{1}(x) \in \mathbb{Z}_{p^{2}}[x]$. But, we know that $\mathbb{Z}_{p^{2}}$ is Armendariz by Proposition (2.1) of \cite{M}. Therefore, $\gamma \delta = 0$, for each $\gamma \in C_{f(x)}$ and $\delta \in C_{g(x)}$. Hence $N$ is an Armendariz ring.\\
If $N^{+}$ is noncyclic. Then by Theorem 2.3.3 of \cite{R}, there is a basis $\{a, b\}$ for $N$ such that $char ~a = char ~b = p$ and one of the following conditions holds:\\ $(i)$ $a^{2} = b^{2} = ab = ba = 0$ $(ii)$ $a^{2} = b, a^{3} = 0$. \\
To prove $N$ is an Armendariz ring, let $f(x), g(x) \in N[x]$ such that $f(x)g(x) = 0$. If $N$ satisfies the case $(i)$ $a^{2} = b^{2} = ab = ba = 0$, then $N = \{0, a, 2a, \ldots, (p-1)a, b, 2b, \ldots, (p-1)b, a+b, a+2b, \ldots, a+(p-1)b, 2a+b, 2a+2b, \ldots, 2a+(p-1)b,\ldots, (p-1)a+b, (p-1)a+2b, \ldots, (p-1)a+(p-1)b \}$. In this case, $\gamma \delta = 0$ for each $\gamma \in C_{f(x)}$ and $\delta \in C_{g(x)}$, because product of any two elements of $N$ is zero.\\
When $N$ satisfies the condition (ii) $a^{2} = b$, $a^{3} = 0$, then $N = \{ 0, a, 2a, \ldots, (p-1)a, a^{2}, 2a^{2}, \ldots,\\ (p-1)a^{2}, a+a^{2}, a+2a^{2}, \ldots, a+(p-1)a^{2}, 2a+a^{2}, 2a+2a^{2}, \ldots, 2a+(p-1)a^{2}, (p-1)a+a^{2}, (p-1)a+2a^{2}, \ldots, (p-1)a+(p-1)a^{2}\}$. In this case, only possibility for the product of two non-zero elements of $N$ is zero in which one of elements in the product is from the set $\{ a^{2}, 2a^{2}, \ldots, (p-1)a^{2} \}$.  Here, we can write $f(x) =a^{m}f_{1}(x), ~g(x) = a^{n}g_{1}(x)$ and $f(x)g(x) = 0$
 when $m+n\geq 3$. Thus, $f(x)g(x) = 0$ implies $\gamma \delta = 0$, for each $\gamma \in C_{f(x)}$ and $\delta \in C_{g(x)}$.
\end{proof}

\begin{pro}
Let $K$ be an ideal of a ring $R$ such that every element of $R\setminus K$ is regular and $K^{2} = 0$. Then $R$ is a weak ideal-Armendariz ring.
\end{pro}
\begin{proof} Since each element of $R\setminus K$ is regular and $K^{2} = 0$, therefore by Proposition (3.7) of \cite{NNN}, $R$ is an Armendariz ring. To prove $R$ is weakly semicommutative ring, let $ab = 0$ in $R$. Then $ab + K = K$. Since $R/ K$ is a domain, so $a \in K$ or $b \in K$. Also, $arb \in K$ for any $r \in R$ and $K^{2} = 0$, therefore $(arb)^{2} = 0$. Hence, $R$ is a weakly semicommutative ring. Thus, $R$ is a weak ideal-Armendariz ring.
\end{proof}

Given a ring $R$ and a bimodule $_{R}M_{R}$, the trivial extension of $R$ by $M$ is the ring $T(R, M)$ with the usual addition and multiplication defined by
$$(r_{1}, m_{1})(r_{2}, m_{2}) = (r_{1}r_{2}, r_{1}m_{2}+m_{1}r_{2}).$$
This is isomorphic to the ring of all matrices of the form $\left(
                                                 \begin{array}{cc}
                                                   r & m \\
                                                   0 & r \\
                                                 \end{array}
                                               \right)$ with usual addition and multiplication of  matrices, where $r\in R$ and $m\in M$.

\begin{pro} For a ring $R$, the following are equivalent:
\begin{itemize}
\item[$(1)$]  $R$ is a reduced ring.
\item[$(2)$]  $R[x]/<x^{n}>$ is a weak ideal-Armendariz ring, where $<x^{n}>$ is the ideal generated by $x^{n}$ for any positive integer $n$.
\item[$(3)$]  The trivial extension $T(R, R)$ is a weak ideal-Armendariz ring.
\end{itemize}
\end{pro}
\begin{proof} $(1)\Rightarrow(2)$: By Theorem 5 of \cite{D}, $R[x]/<x^{n}>$ is an Armendariz ring. Also, reduced ring is weakly semicommutative ring, therefore $R[x]/<x^{n}>$ is a weak ideal-Armendariz ring.
$(2)\Rightarrow (3)$ and $(3)\Rightarrow(1)$ are obvious.
\end{proof}

\begin{pro} Let $I$ be a reduced ideal of a ring $R$. If $R/I$ is a weak ideal-Armendariz ring, then $R$ is a weak ideal-Armendariz ring.
\end{pro}
\begin{proof} Let $ab = 0$ for $a, b \in R$. Since $I$ is reduced ideal of $R$, so $bIa = 0$. Again, $ab = 0$ implies $ab \in I$. Since $R/I$ is a weakly semicommutative, so $(aRb+I)\subseteq N(R/I)$. Therefore, $(arb)^{n} \in I$ for any $r\in R$. Moreover, $(arb)^{n}I(arb)^{n}I \subseteq I$ and this implies $(arb)^{n-1}(arb)I(arb)(arb)^{n-1}I = 0$. This shows that $((arb)^{n}I)^{2} = 0$. Since $I$ is reduced, so $(arb)^{n}I = 0$ and $(arb)^{n}\in (arb)^{n}I$ implies $(arb)^{n} = 0$. Therefore, $(arb) \in N(R)$ for any $r \in R$. Hence, $R$ is a weakly semicommutative ring. Also by Theorem $11$ of \cite{C}, $R$ is an Armendariz ring. Thus, $R$ is a weak ideal-Armendariz ring.
\end{proof}
\begin{pro} Finite direct product of weak ideal-Armendariz rings is weak ideal-Armendariz.
\end{pro}
\begin{proof} Let $R_{1}, R_{2}, \ldots, R_{n}$ be Armendariz rings and $R = \prod _{i=0}^{n}R_{i}$. Let $f(x) = a_{0}+a_{1}x+a_{2}x^{2}+\cdots+a_{n}x^{n}$ and $g(x) = b_{0}+b_{1}x+b_{2}x^{2}+\cdots+b_{m}x^{m} \in R[x]$ such that $f(x)g(x) = 0$, where $a_{i} = (a_{i1}, a_{i2}, \ldots, a_{in})$ and $b_{i} = (b_{i1}, b_{i2}, \ldots, b_{im})$. Define\\
\begin{eqnarray*}
f_{p}(x) = a_{0p}+a_{1p}x+\cdots+a_{np}x^{n},~~ g_{p}(x) = b_{0p}+b_{1p}x+\cdots+b_{mp}x^{m}.
\end{eqnarray*}
By $f(x)g(x) = 0$, we have $a_{0}b_{0} = 0$, $a_{0}b_{1}+a_{1}b_{0} = 0$, $a_{0}b_{2}+a_{1}b_{1}+a_{2}b_{0} = 0,\ldots, a_{n}b_{m} = 0$ this implies

\begin{eqnarray*}
a_{01}b_{01} = 0, a_{02}b_{02} = 0, \ldots a_{0n}b_{0n} = 0\\
a_{01}b_{11}+a_{11}b_{01} = 0, \ldots, a_{0n}b_{1n}+a_{1n}b_{0n} = 0\\
a_{n1}b_{m1} = 0, a_{n2}b_{m2} =0, \ldots, a_{nn}b_{nm} = 0.
\end{eqnarray*}
Therefore, $f_{p}(x)g_{p}(x) = 0$ in $R_{p}[x]$, for $1\leq p \leq n$. Since each $R_{p}$ is an Armendariz ring, $a_{ip}b_{jp} = 0$ for each $i, j$ and $1\leq p \leq n$ and hence $a_{i}b_{j} = 0$ for each $0 \leq i \leq n$ and $0\leq j \leq m$. Thus, $R$ is a weak ideal-Armendariz ring.\\
It is known that $N(\prod_{i=0}^{n}R_{i}) = \prod_{i=0}^{n}N(R_{i})$. \\
Let $(a_{1}, a_{2},\ldots,a_{n}), (b_{1}, b_{2},\ldots,b_{n})\in N(\prod_{i=0}^{n}R_{i})$ such that  $(a_{1}, a_{2},\ldots,a_{n})(b_{1}, b_{2},\ldots,b_{n}) = 0$. Then $a_{i}b_{j} = 0$ for each $1 \leq i, j \leq n$. Since $R_{i}$ is a weakly semicommutative ring, therefore $a_{i}R_{i}b_{j}\subseteq N(R_{i})$ for each $i, j$. So, there exists a maximal positive integer $k$ among all positive integers (all nilpotency) such that $a_{i}R_{i}b_{i} \subseteq N(R_{i})$ for each $i = 1, 2, \ldots, n$. Therefore, $(a_{1}, a_{2},\ldots,a_{n})\prod_{i=0}^{n}R_{i}(b_{1}, b_{2},\ldots,b_{n}) \in \prod_{i=0}^{n}N(R_{i}) = N(\prod_{i=0}^{n}R_{i})$.\\
\end{proof}

Let $R$ be a ring and let $S^{-1}R = \{u^{-1}a ~|~ u \in S, a \in R\}$ with $S$ a multiplicative closed subset of the ring $R$ consisting of central regular elements. Then $S^{-1}R$ is a ring.

\begin{pro}\label{pro8} Let $S$ be multiplicative closed subset of a ring $R$ consisting of central regular elements. Then $R$ is a weak ideal-Armendariz if and only if $S^{-1}R$ is a weak ideal-Armendariz.
\end{pro}
\begin{proof} Since $R$ is an Armendariz ring if and only if $S^{-1}R$ is an Armendariz ring. Also, by Proposition $(3.1)$ of \cite{L}, $S^{-1}R$ is a weakly semicommutative ring. Hence, $S^{-1}R$ is a weak ideal-Armendariz ring.
It is known that subring of a weak ideal-Armendariz ring is weak ideal-Armendariz, therefore converse is true.
\end{proof}
\begin{cor} For any ring $R$, following conditions are equivalent:
\begin{itemize}
\item[$(1)$]$R[x]$ is weak ideal-Armendariz ring.
\item[$(2)$] $R[x, x^{-1}]$ is weak ideal-Armendariz ring.
\end{itemize}
\begin{proof} $(1)\Leftrightarrow (2)$ Clearly, $S = \{1, x, x^{2},\ldots\}$ is a multiplicative set in $R[x]$ and $R[x, x^{-1}] = S^{-1}R[x]$. Rest follows from Proposition \ref{pro8}.
\end{proof}
\end{cor}

For a semiprime ring $R$, note that $R$ is reduced if and only if $N_{*}(R) = N(R)$ if and only if $R$ is semicommutative ring if and only if $R$ is weakly semicommutative ring. Since, for a semiprime ring $R$, $\{0\} = N_{*}(R) = N(R)$.
A ring $R$ is said to be a right Ore ring if for $a, b \in R$ with $b$ is a regular, there exist $a_{1}, b_{1} \in R$ with $b_{1}$ is regular such that $ab_{1} = ba_{1}$. It is known that a ring $R$ is a right Ore ring if and only if the classical right quotient ring $Q(R)$ of $R$ exists.
\begin{pro} Let $R$ be a semiprime right Goldie ring. Then following statements are equivalent:
\begin{itemize}
\item[$(1)$]  $R$ is an ideal-Armendariz ring.
\item[$(2)$]  $R$ is an Armendariz ring.
\item[$(3)$]  $R$ is reduced ring.
\item[$(4)$]  $R$ is semicommutative ring.
\item[$(5)$]  $R$ is weakly semicommutative ring.
\item[$(6)$]  $Q(R)$ is an ideal Armendariz.
\item[$(7)$]  $Q(R)$ is reduced ring.
\item[$(8)$]  $Q(R)$ is semicommutative ring.
\item[$(9)$]  $Q(R)$ is weakly semicommutative ring.
\item[$(10)$]  $R$ is weak ideal-Armendariz ring.
\item[$(11)$]  $Q(R)$ is weak ideal-Armendariz ring.
\end{itemize}
\end{pro}
\begin{proof} Implication from (1) to (9) are obvious.\\
$(9)\Rightarrow(10)$: $Q(R)$ is weakly semicommutative so is $R$ (being subring). Moreover, $Q(R)$ is semicommutative, therefore,  $R$ is an Armendariz by Corollary 13 of [\cite{C}].\\ $(10)\Rightarrow(11)$: Since $R$ is a semiprime right Goldie ring.  Therefore, Armedariz ring and weakly semicommutative both are equivalent and from Theorem 12 of \cite{C}, $R$ is an Armendariz if and only if $Q(R)$ is an Armendariz ring.\\
$(11)\Rightarrow(1)$: It is obvious.
\end{proof}

For an algebra $R$ over commutative ring $S$, the Dorroh extension of $R$ by $S$ is an abelian group $D = R\oplus S$ with multiplication defined by $(r_1, s_1)(r_2, s_2) = (r_1r_2+s_1r_2+s_2r_1, s_1s_2)$, where $r_1, r_2 \in R$ and $s_1, s_2 \in S$.

\begin{thm} Let $R$ be an algebra over a commutative domain $S$ and $D$ be a Dorroh extension of $R$ by $S$. Then $R$ is a weak ideal Armendariz if and only if $D$ is a weak ideal Armendariz ring.
\end{thm}
\begin{proof} Let $D$ be a weak ideal-Armendariz ring. Since $D$ is an extension of $R$, therefore $R$ is a weak ideal-Armendariz ring.\\
To prove the converse, let $R$ be a weak ideal-Armendariz ring. Take $(r_{1}, s_{1})(r_{2}, s_{2}) = 0$. Then $r_{1}r_{2}+s_{1}r_{2}+s_{2}r_{1} = 0$ and $s_{1}s_{2} = 0$.\\

\textbf{Case (1)} If $s_{1} = 0$, then $r_{1}r_{2}+s_{2}r_{1} = 0$. This implies $r_{1}(r_{2}+s_{2}) = 0$, and hence by definition of weakly semicommutative ring $r_{1}R(r_{2}+s_{2})\subseteq N(R)$. Therefore, $(r_{1}, 0)(a, b)(r_{2}, s_{2}) = (r_{1}a+br_{1}, 0)(r_{2}, s_{2}) = (r_{1}(a+b), 0)(r_{2}, s_{2}) = (r_{1}(a+b)r_{2}+r_{1}(a+b)s_{2}, 0) = (r_{1}(a+b)(r_{2}+s_{2}), 0) \subseteq N(R)$. Thus, $(r_{1}, s_{1})D(r_{2}, s_{2}) \subseteq N(D)$.\\

\textbf{Case (2)} If $s_{2} = 0$, then $r_{1}r_{2}+s_{1}r_{2} = 0$, so $(r_{1}+s_{1})r_{2} = 0$. Therefore, $(r_{1}+s_{1})Rr_{2} \subseteq N(R)$. Also, $(r_{1}, s_{1})(a, b)(r_{2}, 0) = (r_{1}a+s_{1}a+r_{1}b, s_{1}b)(r_{2}, 0) = (r_{1}ar_{2}+s_{1}ar_{2}+r_{1}br_{2}+s_{1}br_{2}, 0) = (r_{1}(a+b)r_{2}+s_{1}(a+b)r_{2}, 0) = ((r_{1}+s_{1})(a+b)r_{2}, 0) \subseteq N(R)$. This implies $(r_{1}, s_{1})(a, b)(r_{2}, s_{2}) \subseteq N(D)$. Hence, $D$ is a weakly semicommutative ring. By Theorem (3.4) of \cite{TT}, $D$ is an Armendariz ring. Thus, $D$ is a weak ideal-Armendariz ring.
\end{proof}

A ring $R$ is abelian if every idempotent element is central. In 1998, Anderson and Camillo proved that Armendariz rings are abelian\cite{D}, but weakly semicommutative ring is not abelian, therefore, weak ideal-Armendariz ring is not abelian. In this regards we have the following:

\begin{ex}\label{ex3} Let $D$ be a domain and $R_{1}$, $R_{2}$ be two rings such that $R_{1} = \left(
                                                                                      \begin{array}{cc}
                                                                                        D & D \\
                                                                                        0 & 0 \\
                                                                                      \end{array}
                                                                                    \right)$, $R_2 = \left(
                                                                                                       \begin{array}{cc}
                                                                                                         0 & D \\
                                                                                                         0 & D \\
                                                                                                       \end{array}
                                                                                                     \right)$. Let $0 \neq A = \left(
                                                                                                                      \begin{array}{cc}
                                                                                                                        a_1 & b_1 \\
                                                                                                                        0 & 0 \\
                                                                                                                      \end{array}
                                                                                                                    \right)$, $0 \neq B = \left(
                                                                                                                                     \begin{array}{cc}
                                                                                                                                       a_2 & b_2 \\
                                                                                                                                       0 & 0 \\
                                                                                                                                     \end{array}
                                                                                                                                   \right) \in R_1$ such that $AB = 0$. Then $a_1 a_2 =0$ and $a_1 b_2= 0$. It is only possible when $a_1 = 0$. This implies $ACB \in N(R_1)$ for each $C \in R_1$. Hence, $R_{1}$ is a weakly semicommutative ring. Since, $e_{11}e_{12} = e_{12}$ and $e_{12}e_{11} = 0$, whereas $e_{11}^{2} = e_{11}$, therefore $R_{1}$ is a non-abelian ring.\\
Similarly, we can prove that $\left(
                                \begin{array}{cc}
                                  0 & D \\
                                  0 & D \\
                                \end{array}
                              \right)$ is a non-abelian weakly semicommutative ring.

Also, from Example (2.14) of \cite{TT}, $R_1$ and $R_2$ both are Armendariz rings. Thus, $R_1$ and $R_2$ both are non-abelian weak ideal-Armendariz rings.
\end{ex}

\section{STRONGLY NIL-IFP}

In this section we have introduced the concept of strongly nil-IFP which is the generalisation of weak ideal-Armendariz ring. Towards this, we have the following definition:\\

\emph{A ring $R$ is said to be strongly nil-IFP if for any $f(x), g(x) \in R[x]$ such that $f(x)g(x) = 0$, then $arb \in N(R)$, for all $a \in C_{f(x)}, b \in C_{g(x)}$ and each $r \in R$.} \\

By definition, it is clear that weak ideal-Armendariz ring is strongly nil-IFP but converse is not true.\\
\begin{ex} Suppose $R = \mathbb{Z}_{7}[x, y]/<x^{3}, x^{2}y^{2}, y^{3}>$, where $\mathbb{Z}_{7}$ is a Galois field of order $7$, $\mathbb{Z}_{7}[x, y]$ is the polynomial ring with two indeterminates $x$ and $y$ over $\mathbb{Z}_{7}$ and $(x^{3}, x^{2}y^{2}, y^{3})$ is the ideal of $\mathbb{Z}_{7}[x, y]$ generated by $x^{3}, x^{2}y^{2}, y^{3}$. Let $R[t]$ be the polynomial ring over $R$ with an indeterminate $t$.
 We consider $f(t) = {\bar x} + {\bar y}t$, $g(t) = 3{\bar x^{2}} + 4{\bar x \bar y}t + 3{\bar y^{2}}t^{2} \in R[t]$. Then $f(t)g(t) = 0$ but ${\bar x}.4{ \bar x }{\bar y} \neq {\bar 0}$. Therefore, $R$ is not an Armendariz ring, hence it is not a weak ideal-Armendariz ring. But $R$ is a strongly nil-IFP, because for any $f(t), ~g(t) \in R[t]$ if $f(t)g(t) = 0$, then $arb \in N(R)$ for each $r \in R$, $a \in C_{f(t)}$ and $b \in C_{g(t)}$.
\end{ex}
\begin{ex}(\cite{C}, Example (2))\label{ex2} Let $\mathbb{Z}_{2}$ be the field of integers modulo 2 and $A = \mathbb{Z}_{2}[a_{0}, a_{1}, a_{2}, b_{0}, b_{1}, b_{2}, c]$ be the free algebra of polynomials with zero constant terms in noncommuting indeterminates $a_{0}, a_{1}, a_{2}, b_{0}, b_{1}, b_{2}, c$ over $\mathbb{Z}_{2}$. Here, $A$ is a ring without identity. Consider an ideal $I$ of $\mathbb{Z}_{2}+A$, generated by \\
\begin{eqnarray*}
a_{0}b_{0}, a_{0}b_{1}+a_{1}b_{0}, a_{0}b_{2}+a_{1}b_{1}+a_{2}b_{0}, a_{1}b_{2}+a_{2}b_{1}, a_{2}b_{2}, a_{0}rb_{0}, a_{2}rb_{2}\\
(a_{0}+a_{1}+a_{2})r(b_{0}+b_{1}+b_{2}), r_{1}r_{2}r_{3}r_{4}
\end{eqnarray*}
where $r, r_{1}, r_{2}, r_{3}, r_{4} \in A$. Then clearly, $A^{4} \subseteq I$. Now, let $R = (\mathbb{Z}_{2}+A)/I$. By Example (2) of \cite{C} $R$ is semicommutative ring. Therefore, by Proposition \ref{pro4}, $R$ is a strongly nil-IFP. If we take $f(t) = a_{0}+a_{1}t+a_{2}t^{2}, g(t) = b_{0}+b_{1}t+b_{2}t^{2} \in R[t]$, then $f(t)g(t) = 0$ but $a_{0}b_{1} \neq 0$. So, $R$ is not an Armendariz ring. Thus, $R$ is not a weak ideal Armendariz ring.
\end{ex}
\begin{pro}\label{pro9}
\begin{itemize}
\item[$(1)$] The classes of strongly nil-IFP is closed under subrings.
\item [$(2)$] The classes of strongly nil-IFP is closed under direct sums.
\end{itemize}
\end{pro}
\begin{proof}
\begin{itemize}
\item[$(1)$] It is obvious that subring of strongly nil-IFP is strongly nil-IFP.
\item[$(2)$] Suppose that $R_{\alpha}$ is strongly nil-IFP for each $\alpha \in \Gamma$ and let $R = \bigoplus_{\alpha \in \Gamma} R_{\alpha}$. If $f(x)g(x) = 0 \in R[x]$, then $f_{\alpha}(x)g_{\alpha}(x) = 0$, where $f_{\alpha}(x), g_{\alpha}(x) \in R_{\alpha}(x)$ for each $\alpha \in \Gamma$. Since each $R_{\alpha}$ is strongly nil-IFP, $a_{\alpha}r_{\alpha}b_{\alpha} \in N(R_{\alpha})$ for each $a_{\alpha} \in C_{f_{\alpha}}$ and $b_{\alpha} \in C_{g_{\alpha}}$ and for all $r_{\alpha} \in R_{\alpha}$. We know that $N(R) = \bigoplus_{\alpha \in \Gamma}N(R_{\alpha})$, therefore $R$ is a strongly nil-IFP.
\end{itemize}
\end{proof}
\begin{pro}\label{pro4} Every semicommutative ring is strongly nil-IFP.
\end{pro}
\begin{proof} It is known that semicommutative rings are weak Armendariz. So, for any $f(x)$ and $g(x) \in R[x]$ such that $f(x)g(x) = 0$ implies $ab\in N(R)$ for each $a \in C_{f(x)}$ and $b \in C_{g(x)}$. Since $R$ is a semicommutative ring, therefore $ab \in N(R)$ implies $arb \in N(R)$ for each $r \in R$. Hence, $R$ is a strongly nil-IFP.
\end{proof}
Converse of  above is not true. See the below given example:
\begin{ex} Construction of example is same as Example \ref{ex2}. Here, consider the ideal $I$ of $\mathbb{Z}_{2}+A$ which is generated by \\
\begin{eqnarray*}
a_{0}b_{0}, a_{0}b_{1}+a_{1}b_{0}, a_{0}b_{2}+a_{1}b_{1}+a_{2}b_{0}, a_{1}b_{2}+a_{2}b_{1}, a_{2}b_{2},\\
(a_{0}+a_{1}+a_{2})r(b_{0}+b_{1}+b_{2}), r_{1}r_{2}r_{3}r_{4}
\end{eqnarray*}
where $r, r_{1}, r_{2}, r_{3}, r_{4} \in A$. Then, $A^{4} \subseteq I$. Now, let $R = (\mathbb{Z}_{2}+A)/I$. Since, $a_{0}b_{0} = 0$ and $a_{0}b_{2}b_{0} \neq 0$, therefore $R$ is not a semicommutative ring. Also, whenever, $f(t)g(t)  \in R[t]$ such that $f(t)g(t) = 0$, then $aRb \subseteq N(R)$ for any and each $a \in C_{f(t)}$ and  $b \in C_{g(t)}$. Thus, $R$ is a strongly nil-IFP.
\end{ex}
\begin{pro}
Every strongly nil-IFP ring is a weakly semicommutative ring.
\end{pro}
\begin{proof} It is obvious.
\end{proof}
 Below given two examples show that weakly semicommutative ring/ Armendariz ring is not strongly nil-IFP. \\
\begin{ex} Let $K$ be a field and $K[a_{0}, a_{1}, b_{0}, b_{1}]$ be the free algebra with noncommuting indeterminates $a_{0}, a_{1}, b_{0}, b_{1}$ over $K$. Let $I$ be an ideal of $K[a_{0}, a_{1}, b_{0}, b_{1}]$ generated by\\

 $a_{0}b_{0}, a_{0}b_{1}+a_{1}b_{0}, a_{1}b_{1}, (r_{1}a_{0}r_{2}b_{0}r_{3})^{2}, (r_{4}a_{1}r_{5}b_{4}r_{6})^{2},$ $r_{7}(a_{0}+a_{1})r_{8}(b_{0}+b_{1})r_{9}$,\\

where $r_{1}, r_{2}, r_{3}, r_{4}, r_{5}, r_{6}, r_{7}, r_{8}, r_{9} \in K[a_{0}, a_{1}, b_{0}, b_{1}]$. Next, let $R = K[a_{0}, a_{1}, b_{0}, b_{1}]/I$. To prove $R$ is not a strongly nil IFP, consider $f(x) = a_{0}+a_{1}x$ and $g(x) = b_{0}+b_{1}x$. Then $f(x)g(x) = a_{0}b_{0}+(a_{0}b_{1}+a_{1}b_{0})x+a_{1}b_{1}x^{2} = 0$ but $a_{0}a_{0}b_{1} \notin N(R)$. Therefore, $R$ is not a strongly nil IFP. By construction of ideal, it is clear that $R$ is a weakly semicommutative ring.
\end{ex}
\begin{ex}\label{ex1} Let $K$ be a field and $R = K[a, b]/<a^{2}>$. Then $R$ is an Armendariz ring by Example (4.8) of \cite{RR}. Therefore $R$ is a weak Armendariz ring. Here, if we take $f(t) = ba + ba t$, $g(t) = a + at$, then $f(t)g(t) = 0$ but $(ba)b(a)$ is not nilpotent. Thus, $R$ is not a strongly nil-IFP.
\end{ex}

\begin{pro}\label{pro5} If $R$ is Armendariz and $QRPR$, then $R$ is strongly nil-IFP.
\end{pro}
\begin{proof} Let $f(x) = \sum_{i=0}^{m}a_{i}x^{i}$ and $g(x) = \sum_{j=0}^{n}b_{j}x^{j}$ $\in R[x]$ such that $f(x)g(x) = 0$. Then $a_{i}b_{j} = 0$ for each $i, j$. Also, $R$ is quasi-reversible-over-prime-radical, therefore $b_{j}a_{i} \in N_{*}(R)$ and this implies $b_{j}a_{i}R \subseteq N_{*}(R)$ and $b_{j}a_{i}R \subseteq N(R)$. Hence, $a_{i}Rb_{j}\subseteq N(R)$ for each $i, j$. Thus, $R$ is a strongly nil-IFP.
\end{proof}
\begin{pro}\label{pro1} Strongly nil-IFP ring is weak Armendariz ring.
\end{pro}
\begin{proof} Let $f(x) = \sum_{i=0}^{m}a_{i}x^{i}$, $g(x) = \sum_{j=0}^{n}b_{j}x^{j} \in R[x]$ such that $f(x)g(x) = 0$. Then $a_{i}Rb_{j}\subseteq N(R)$  for each $i, j$. If $R$ has unity, then $a_{i}b_{j} \in N(R)$ for each $i, j$. If ring has no unity, then $a_{i}(b_{j}a_{i})b_{j} \in N(R)$ for each $i, j$. This implies $(a_{i}b_{j})^{2} \in N(R)$ and hence $a_{i}b_{j} \in N(R)$ for each $i, j$. Thus, $R$ is a weak Armendariz ring.
\end{proof}

\begin{rem}
\begin{itemize}
\item[$(1)$] By Example \ref{ex1}, we can say that converse of above is not true.
\item[$(2)$] Every reversible ring is strongly nil-IFP but converse is not true.
\end{itemize}
\end{rem}

\begin{ex} Let $\mathbb{Z}_{2}$ be the field of integers modulo $2$ and $A =  \mathbb{Z}_{2}[a_{0}, a_{1}, a_{2}, b_{0}, b_{1}, b_{2}]$ be the free algebra of polynomials with zero constant term in noncommuting indeterminates $a_{0}, a_{1}, a_{2}, b_{0}, b_{1}, b_{2}$ over $\mathbb{Z}_{2}$. Here, $A$ is a ring without identity. Consider an ideal of $\mathbb{Z}_{2}+A$, say $I$, generated by $a_{0}b_{0}, a_{0}b_{1}+a_{1}b_{0}, a_{0}b_{2}+a_{1}b_{1}+a_{2}b_{0}, a_{1}b_{2}+a_{2}b_{1}, a_{2}b_{2}$ and $s_{1}s_{2}s_{3}s_{4}s_{5}s_{6}$ where $s_{1}, s_{2}, s_{3}, s_{4}, s_{5}, s_{6} \in A$. Let $R= \frac{\mathbb{Z}_{2}+A}{I}$. Here, $R$ is a strongly nil-IFP. Since, $a_{0}b_{0} = 0$ but $b_{0}a_{0} \neq 0$, therefore $R$ is not a reversible ring.
\end{ex}
\begin{lem} If $R$ is strongly nil-IFP with no non-zero nil ideals, then $R$ is reversible ring.
\end{lem}
\begin{proof} To prove $R$ is reversible ring, let $ab = 0$. Then $ax bx = 0$ and this implies $aRb \subseteq N(R)$. Therefore, $Rba \subseteq N(R)$, hence $Rba$ is nil one sided ideal of $R$. Also, $Rba = \{0\}$, so $ba = 0$. Hence, $R$ is reversible ring.
\end{proof}

\begin{pro}\label{pro7} Let $R$ be a 2-primal ring. If $f(x)R[x]g(x) \in N_{*}(R)[x]$, then $aRb \subseteq N_{*}(R)$ for each $a \in C_{f(x)}$ and $b \in C_{g(x)}$.
\end{pro}
\begin{proof} Let $f(x), g(x) \in R[x]$ such that $f(x)R[x]g(x) \in N_{*}(R)[x]$. Since $R/N_{*}(R)$ is reduced, therefore $R/N_{*}(R)$ is an Armendariz ring and hence $aRb \subseteq N_{*}(R)$ for each $a \in C_{f(x)}$ and $b \in C_{g(x)}$.
\end{proof}
\begin{pro}\label{pro2}
\begin{itemize}
\item[$(1)$] Every 2-primal ring is strongly nil-IFP.
\item[$(2)$] Every NI ring is strongly nil-IFP.
\end{itemize}
\end{pro}
\begin{proof}\begin{itemize}
\item[$(1)$] Let $f(x) = \sum_{i=0}^{m}a_{i}x^{i}$ and $g(x) = \sum_{j=0}^{n}b_{j}x^{j}\in R[x]$ such that $f(x)g(x) = 0$. Then $\overline{f(x)}\overline{g(x)} = \overline{0}$ in $(R/N_{*}(R))[x]$. Since, $R/N_{*}(R)$ is reduced ring, therefore $R/N_{*}(R)$ is an Armendariz ring. This implies $a_{i}b_{j}+N_{*}(R) = N_{*}(R)$, $a_{i}b_{j} \in N_{*}(R)$ for each $i, j$ and hence $b_{j}a_{i} \in N_{*}(R)$, so $b_{j}a_{i}R \in N_{*}(R)$. This implies $a_{i}Rb_{j} \in N_{*}(R)$. Thus, $R$ is a strongly nil-IFP.
\item[$(2)$] It is obvious.
\end{itemize}
\end{proof}

\begin{pro} Let $R$ be a ring of bonded index of nilpotency $2$. Then the following conditions are equivalent:
\begin{itemize}
\item[$(1)$] $R$ is strongly nil-IFP.
\item[$(2)$] $R$ is NI.
\end{itemize}
\end{pro}
\begin{proof}
$(1)\Rightarrow (2)$ Let $a, b \in N(R)$. Then $a^{2} = 0, b^{2} = 0$. Now, $(a+b)^{2} = a^{2}+ab+ba+b^{2} = ab + ba$ and $(a+b)^{4} = abab +baba = 0$ and so $(a+b)^{2} = 0$. Thus, $a+b \in N(R)$.\\
Again, $ab = 0$ implies $Rab = 0$ and hence $RaRb \subseteq N(R)$. In particular, $a^{2} = 0$, then $RaRa \subseteq N(R)$ and hence $Ra\subseteq N(R)$. Thus, $ar, ra \in N(R)$.

$(2)\Rightarrow (1)$ It is obvious by Proposition \ref{pro2}.
\end{proof}

\begin{pro} For a ring $R$, let $R/I$ be strongly nil IFP. If $I$ is a semicommutative ideal of $R$, then $R$ is a strongly nil IFP.
\end{pro}
\begin{proof} Let $f(x) = \sum_{i=0}^{m}a_{i}x^{i}$ and $g(x) = \sum_{j=0}^{n}b_{i}x^{j}\in R[x]$ be such that $f(x)g(x) = 0$. This implies
\begin{equation} \label{1}
\sum_{k=0}^{m+n}(\sum_{i+j=k}a_{i}b_{j})x^{k} = 0.
\end{equation}
Hence, we have the following equations
\begin{equation}\label{2}
  \sum_{i+j = l}a_{i}b_{j} = 0,~~~l = 0, 1, 2,\ldots, m+n.
\end{equation}
Also, $f(x)g(x) = 0$ implies that $\overline{f(x)}\overline{g(x)} = \overline{0} \in (R/I)[x]$. Since $R/I$ is strongly nil IFP, so $\overline{a_{i}}\overline{R}~\overline{b_{j}} \subseteq N(R/I)$ for each $i, j$ for $0 \leq i\leq m$ and $0 \leq j \leq n$. This implies, there exists a positive integer $n_{ij}$ such that $(a_{i}rb_{j})^{n_{ij}}\in I$ for each $i, j$ and all $r \in R$. Now, we use principle of induction on $i+j$ to prove $a_{i}rb_{j} \in N(R)$.\\

If $i+j = 0$, we have $a_{0}b_{0} = 0$. Also from above, there exists a positive integer $p$ such that $(a_{0}rb_{0})^{p} \in I$. This implies $(a_{0}rb_{0})^{p}a_{0}, b_{0}(a_{0}rb_{0})^{p} \in I$. Therefore, $(a_{0}rb_{0})^{p}a_{0}b_{0}(a_{0}rb_{0})^{p} = 0$. Since, $rb_{0}(a_{0}rb_{0})^{p}, (a_{0}rb_{0})^{p}a_{0}r \in I$, so $(a_{0}rb_{0})^{p}a_{0}rb_{0}(a_{0}rb_{0})^{p}b_{0}\\(a_{0}rb_{0})^{p} = 0$ and again $(a_{0}rb_{0})^{p}a_{0}rb_{0}(a_{0}rb_{0})^{p}(a_{0}rb_{0})^{p}a_{0}r b_{0}(a_{0}rb_{0})^{p} = 0$, hence $(a_{0}rb_{0})^{4p+2} = 0$. Thus, $a_{0}Rb_{0} \subseteq N(R)$.\\

Let the result is true for all positive integers less than $l$, i.e. $a_{i}Rb_{j} \subseteq N(R)$, when $i+j<l$.
Let there exists a positive integer $q$ such that $(a_{0}rb_{l})^{q} \in I$ for a fixed $r \in R$. Also, by assumption, we have $a_{0}Rb_{l-1} \subseteq N(R)$, so $b_{l-1}a_{0} \in N(R)$. Then there exists $s$ such that $(b_{l-1}a_{0})^{s} = 0$. Hence, we have
\begin{equation*}
((a_{1}b_{l-1})(a_{0}rb_{l})^{q+1}a_{1})(b_{l-1}a_{0})^{s}(b_{l-1}(a_{0}rb_{l})^{q+1}) = 0.
\end{equation*}
This implies
\begin{equation*}
((a_{1}b_{l-1})(a_{0}rb_{l})^{q+1}a_{1})(b_{l-1}a_{0})(b_{l-1}a_{0})^{s-1}(b_{l-1}(a_{0}rb_{l})^{q+1}) = 0.
\end{equation*}
Since $rb_{l}(a_{0}rb_{l})^{q}a_{1} \in I$ and $I$ is semicommutative, then
\begin{equation*}
((a_{1}b_{l-1})(a_{0}rb_{l})^{q+1}a_{1})(b_{l-1}a_{0})(rb_{l}(a_{0}rb_{l})^{q}a_{1})(b_{l-1}a_{0})^{s-1}(b_{l-1}(a_{0}rb_{l})^{q+1}) = 0.
\end{equation*}
This implies,
\begin{equation*}
(((a_{1}b_{l-1})(a_{0}rb_{l})^{q+1})^{2}a_{1})(b_{l-1}a_{0})(b_{l-1}a_{0})^{s-2}(b_{l-1}(a_{0}rb_{l})^{q+1})= 0.
\end{equation*}
Again, by using semicommutativity of $I$,
\begin{equation*}
(((a_{1}b_{l-1})(a_{0}rb_{l})^{q+1})^{2}a_{1})(b_{l-1}a_{0})rb_{l}(a_{0}rb_{l})^{q}a_{1}(b_{l-1}a_{0})^{s-2}(b_{l-1}(a_{0}rb_{l})^{q+1})= 0.
\end{equation*}
Continuing this process, we get $((a_{1}b_{l-1})(a_{0}rb_{l})^{q+1})^{s+2} = 0$ and hence $((a_{1}b_{l-1})(a_{0}rb_{l})^{q+1})\\ \in N(I)$. \\
Similarly, we can show $(a_{i}b_{l-i})(a_{0}rb_{l})^{q+1} \in N(I)$, for $2 \leq i \leq l$. By equation (\ref{2}), we have
\begin{equation}\label{3}
a_{0}b_{l}+a_{1}b_{l-1}+a_{2}b_{l-2}+\cdots+a_{l}b_{0} = 0.
\end{equation}
Multiplying in equation (\ref{3}) by $(a_{0}rb_{l})^{q+1}$ from right, we get
\begin{equation*}
a_{0}b_{l}(a_{0}rb_{l})^{q+1}+a_{1}b_{l-1}(a_{0}rb_{l})^{q+1}+a_{2}b_{l-2}(a_{0}rb_{l})^{q+1}+\cdots+a_{l}b_{0}(a_{0}rb_{l})^{q+1} = 0.
\end{equation*}
Since $N(I)$ is an ideal and $(a_{i}b_{l-i})(a_{0}rb_{l})^{q+1} \in N(I)$, for $2 \leq i \leq l$, therefore $a_{0}b_{l}(a_{0}rb_{l})^{q+1}\\ \in N(I)$. Also $(a_{0}rb_{l})^{q} \in I$, so $(a_{0}rb_{l})^{q} a_{0}b_{l}(a_{0}rb_{l})^{q+1} \in N(I)$. This implies, $(b_{l}(a_{0}rb_{l})^{q+1})\\(a_{0}rb_{l})^{q} a_{0}) \in N(I)$. Since, $((a_{0}rb_{l})^{q}a_{0}r), (rb_{l}(a_{0}rb_{l})^{q}) \in I$,
therefore $((a_{0}rb_{l})^{q}a_{0}r)(b_{l}(a_{0}rb_{l})^{q+1})\\((a_{0}rb_{l})^{q} a_{0})(rb_{l}(a_{0}rb_{l})^{q}) \in N(I)$ and hence $(a_{0}rb_{l})^{4q+3} \in N(I)$. Hence, $(a_{0}rb_{l})\in N(R)$ and thus $a_{0}Rb_{l}\subseteq N(R)$.\\
Let $(a_{1}rb_{l-1})^{t} \in I$, therefore by above argument $(a_{i}b_{l-i})(a_{1}rb_{l-1})^{t+1} \in N(I)$, for $2 \leq i \leq l$. Since $a_{0}b_{l} \in N(R)$, then $(a_{0}b_{l})^{u} = 0$. Therefore, $(a_{1}rb_{l-1})^{t+1}(a_{0}b_{l})^{u}(a_{1}rb_{l-1})^{t+1} = 0$. This implies,
\begin{equation*}
(a_{1}rb_{l-1})^{t+1}(a_{0}b_{l})(a_{0}b_{l})^{u-1}(a_{1}rb_{l-1})^{t+1} = 0.
\end{equation*}
Again, $I$ is semicommutative, we have
\begin{equation*}
(a_{1}rb_{l-1})^{t+1}(a_{0}b_{l})(a_{1}rb_{l-1})^{t+1}(a_{0}b_{l})^{u-1}(a_{1}rb_{l-1})^{t+1} = 0.
\end{equation*}
\begin{equation*}
(a_{1}rb_{l-1})^{t+1}(a_{0}b_{l})(a_{1}rb_{l-1})^{t+1}(a_{0}b_{l})(a_{0}b_{l})^{u-2}(a_{1}rb_{l-1})^{t+1} = 0.
\end{equation*}
Continuing this process, we get $((a_{1}rb_{l-1})^{t+1}(a_{0}b_{l}))^{u}(a_{1}rb_{l-1})^{t+1} = 0$ and so $((a_{1}rb_{l-1})^{t+1}\\(a_{0}b_{l}))^{u+1} = 0$, therefore $((a_{0}b_{l})(a_{1}rb_{l-1})^{t+1})^{u+2} = 0$. Thus, $(a_{0}b_{l})(a_{1}rb_{l-1})^{t+1} \in N(I)$.\\
Multiplying in equation (\ref{3}) by $(a_{1}rb_{l-1})^{t+1}$ from right side, we get $(a_{1}b_{l-1})(a_{1}rb_{l-1})^{t+1} \in N(I)$. Again, by above analogy, we get $(a_{1}rb_{l-1})^{4t+3} \in N(I)$ and hence $(a_{1}rb_{l-1}) \in N(R)$. Thus, $a_{1}Rb_{l-1} \subseteq N(R)$. Similarly, we can show that $a_{2}Rb_{l-2}, a_{3}Rb_{l-3}, \ldots, a_{l}Rb_{0} \in N(R)$. Therefore, by induction we have $a_{i}Rb_{j} \subseteq N(R)$ for each $i, j$.
\end{proof}
\begin{pro} If $I$ is a nilpotent ideal of a ring $R$ and $R/I$ is strongly nil-IFP, then $R$ is a strongly nil-IFP.
\end{pro}
\begin{proof} Let $f(x)$ and $g(x)\in R[x]$ such that $f(x)g(x) = 0$. Then $\overline{f(x)}\overline{g(x)} = \overline{0}$, therefore $(arb)^{n} \in I$ for each $a \in C_{f(x)}, b \in C_{g(x)}$ and for all $r \in R$, since $R/I$ is strongly nil-IFP. This implies $((arb)^{n})^{m} = 0$, because $I$ is nilpotent ideal with nilpotency $m$. Thus, $R$ is a strongly nil-IFP.
\end{proof}

\begin{thm} Let $R$ be a ring and $n$ be a positive integer. If $R$ is 2-primal, then $R[x]/<x^{n}>$ is a strongly nil-IFP.
\end{thm}
\begin{proof} Result is obvious for $n=1$ by Proposition \ref{pro2}. For $n \geq 2$, we use the technique of the proof of [\cite{D}, Theorem 5]. Here, we denote $\overline{x}$ by $w$ and hence $R[x]/<x^{n}>~ = ~R[w]  = R+Rw+Rw^{2}+\cdots+Rw^{n-1}$, where $w$ commutes with the elements of $R$ and $w^{n} = 0$. Let $f(t) = a_{0}(w)+a_{1}(w)t+a_{2}(w)t^{2}+\cdots+a_{r}(w)t^{r},$ $g(t) = b_{0}(w)+b_{1}(w)t+b_{2}(w)t^{2}+\cdots+b_{s}(w)t^{s} \in R[w][t]$ with $f(t)g(t) = 0$. In order to prove $a_{i}(w)R[w]b_{j}(w) \in N(R[w])$, we can write $a_{i}(w) = a_{i}^{(0)}+a_{i}^{(1)}w+a_{i}^{(2)}w^{2}+\cdots+a_{i}^{(n-1)}w^{n-1}$ and $b_{j}(w) = b_{j}^{(0)}+b_{j}^{(1)}w+b_{j}^{(2)}w^{2}+\cdots+b_{j}^{(n-1)}w^{n-1}$. On the other hand, we can write $f(t) = f_{0}(t)+f_{1}(t)w+f_{2}(t)w^{2}+\cdots+f_{n-1}(t)w^{n-1}$ and $g(t) = g_{0}(t)+g_{1}(t)w+g_{2}(t)w^{2}+\cdots+g_{n-1}(t)w^{n-1}$, where $f_{i}(t) = a_{o}^{(i)}+a_{1}^{(i)}t+a_{2}^{(i)}t^{2}+\cdots+a_{r}^{(i)}t^{r}$ and $g_{j}(t) = b_{o}^{(j)}+b_{1}^{(j)}t+b_{2}^{(j)}t^{2}+\cdots+b_{s}^{(j)}t^{s}$. Now, it is sufficient to proof that $a_{p}^{(i)}w^{i}R[w]b_{q}^{(j)}w^{j} \in N(R[w])$ for all $i, j, p, q$. From $f(t)g(t) = 0$, we have $f_{i}(t)g_{j}(t) = 0$ for $i+j = n$ and this implies $a_{p}^{(i)}w^{i}R[w]b_{q}^{(j)}w^{j} = 0 \in N(R[w])$. If $i+j<n$, then $f_{0}(t)g_{0}(t) = 0$, $f_{0}(t)g_{1}(t)+f_{1}(t)g_{0}(t) = 0$, \ldots, $f_{0}(t)g_{n-1}(t)+f_{1}(t)g_{n-2}(t)+f_{2}(t)g_{n-3}(t)+\cdots+f_{n-1}(t)g_{0}(t) = 0$. Since $R$ is 2-primal, therefore $R[t]$ is 2-primal and hence $R[t]$ is strongly nil-IFP. Therefore, $f_{i}(t)R[t]g_{j}(t) \in N(R[t]) = N(R)[t]$. By Proposition \ref{pro7}, we have $a_{p}^{(i)}Rb_{q}^{(j)} \in N(R)$. Thus, $a_{p}^{(i)}w^{i}R[w] b_{q}^{(j)}w^{j} \in N(R[w])$.
\end{proof}

\begin{pro}\label{pro3} For a ring $R$ and $n \geq 2$, the following are equivalent:
\begin{itemize}
\item[$(1)$] $R$ is strongly nil-IFP.
\item[$(2)$] $U_{n}(R)$ is strongly nil-IFP.
\item[$(3)$] $D_{n}(R)$ is strongly nil-IFP.
\item[$(4)$] $V_{n}(R)$ is strongly nil-IFP.
\item[$(5)$]  $T(R, R)$ is strongly nil-IFP.
\end{itemize}
\end{pro}

\begin{proof} It is sufficient to show $(1)\Rightarrow(2)$. Let $f(x) = A_{0}+A_{1}x+A_{2}x^{2}+\cdots+A_{r}x^{r}$ and $g(x) = B_{0}+B_{1}x+B_{2}x^{2}+\cdots+B_{s}x^{s} \in U_{n}(R)[x]$ such that $f(x)g(x) = 0$, where $A_{i}'s, B{j}'s$ are $n\times n$ upper triangular matrices over $R$ as below:\\

$A_{i} = \left(
           \begin{array}{cccc}
             a_{11}^{i} & a_{12}^{i} & \ldots & a_{1n}^{i} \\
             0 & a_{22}^{i} & \ldots & a_{nn}^{i} \\
             \vdots & \vdots & \ddots & \vdots \\
             0 & 0 & \ldots & a_{nn}^{i} \\
           \end{array}
         \right)$,~~~~~
         $B_{j} = \left(
           \begin{array}{cccc}
             b_{11}^{j} & b_{12}^{j} & \ldots & b_{1n}^{j} \\
             0 & b_{22}^{j} & \ldots & b_{nn}^{i} \\
             \vdots & \vdots & \ddots & \vdots \\
             0 & 0 & \ldots & b_{nn}^{j} \\
           \end{array}
         \right) \in U_{n}(R)$.\\\\
Then, from $f(x)g(x) = 0$, we have $(\sum _{i=0}^{r}a_{pp}^{i}x^{i})$ $(\sum_{j=0}^{s}b_{pp}^{j}x^{j}) = 0\in R[x]$, for $p = 1, 2\ldots n$. Since $R$ is strongly nil-IFP, therefore $a_{pp}^{i}rb_{pp}^{j} \in N(R)$, for each $p$, for each $i$, $j$ and for all $r \in R$. This implies $(a_{pp}^{i}rb_{pp}^{j})^{m_{ijp}} = 0$ for each $p$ and $i$, $j$ and for all $r \in R$. Let $m_{ij} = m_{ij1} m_{ij2}...m_{ijn}$. Therefore,
        $(A_{i}CB_{j})^{m_{ij}} = \left(
                                   \begin{array}{cccc}
                                     0 & \ast & \ldots & \ast \\
                                     0 & 0 & \ldots & \ast \\
                                     \vdots & \vdots & \ddots & \vdots \\
                                     0 & 0 & \ldots & 0 \\
                                   \end{array}
                                 \right)$ and $((A_{i}CB_{j})^{m_{ij}})^{n} = 0$, for each $C \in U_{n}(R)$ and for each $i, j$. Hence, $A_{i}CB_{j} \in N(U_{n}(R))$ for each $i, j$, where $0 \leq i \leq s$, $0 \leq j \leq t$ and for all $C \in U_{n}(R)$. Thus, $U_{n}(R)$ is strongly nil-IFP.
\end{proof}

\begin{pro}\label{pro6} Let $R$  be a strongly nil-IFP. If $N(R)[x] = N(R[x])$, then $R[x]$ is strongly nil-IFP.
\end{pro}
\begin{proof}
 Let $p(y) = f_{0}(x) + f_{1}(x)y + \cdots + f_{m}(x)y^{m}$, $q(y) = g_{0} + g_{1}(x)y +\cdots +g_{n}y^{n} \in R[x][y]$ such that $p(y)q(y) = 0$, where $f_{i}(x), g_{j}(x) \in R[x]$. Write $f_{i}(x) = a_{i0} + a_{i1}x + \cdots + a_{iu_{i}}x^{u_{i}}$, $g_{j}(x) = b_{j0} + b_{j1}x + \cdots + b_{jv_{j}}x^{v_{j}}$, for each $0 \leq i \leq m$ and $0 \leq j \leq n$, where $a_{i0}, a_{i1}, \ldots, a_{iu_{i}}, b_{j0}, b_{j1}, \ldots, b_{jv_{j}} \in R$. Choose a positive integer $k$ such that $k > deg(f_{0}(x)) + deg(f_{1}(x)) + \cdots + deg(f_{m}(x)) + deg(g_{0}(x)) + deg (g_{1}(x)) + \cdots + deg(g_{n}(x))$. Since $p(y)q(y) = 0 \in R[x][y]$, we get

$$\left \{
\begin{array}{ll}
f_{0}(x)g_{0}(x) = 0\\
f_{0}(x)g_{1}(x) + f_{1}(x)g_{0}(x) = 0\\
\ldots\ldots\ldots\ldots\; \; \; \; \;\; \; \; \; \;\; \; \; \; \; \; \; \; \; \; \; \; \; \; \; \; \; \; \; \; \; \; \; \; \; \; \; \; \; \; \; \; \; \; \;\; \; \; \; \;\; \; \; \; \; \\
\ldots\ldots\ldots\ldots\; \; \; \; \;\; \; \; \; \;\; \; \; \; \; \; \; \; \; \; \; \; \; \; \; \; \; \; \; \; \; \; \; \; \; \; \; \; \; \; \; \; \; \; \;\; \; \; \; \;\; \; \; \; \; \\
\ldots\ldots\ldots\ldots\\
f_{m}(x)g_{n}(x) = 0
\end{array}\right.$$
Now, put
\begin{eqnarray*}
p(x^{k})&=&f(x) = f_{0}(x) + f_{1}(x)x^{k} + f_{2}x^{2k} + \cdots + f_{m}(x)x^{mk};\\
q(x^{k})&=&g(x) = g_{0}(x) + g_{1}(x)x^{k} + g_{2}x^{2k} + \cdots + g_{n}x^{nk}. \; \; \; \; \; \; \; \; \; \; \; \; \; \; \; \;(\ast\ast)
\end{eqnarray*}
Then, we have
$$f(x)g(x) = f_{0}(x)g_{0}(x) + (f_{0}(x)g_{1}(x) + f_{1}(x)g_{0}(x))x^{k} + \cdots f_{m}(x)g_{n}(x)x^{(n + k)}.$$
Therefore, by $(\ast\ast)$, we have $f(x)g(x) = 0$ in $R[x]$. On the other hand, we have $f(x)g(x) = (a_{00} + a_{01}x + \cdots +a_{0u_{0}}x^{u_{0}} + a_{10}x^{k} + a_{11}x^{k + 1} + \cdots + a_{1u_{1}}x^{k + u_{1}} + \cdots + a_{m0} + a_{m1}x^{mk + 1} + \cdots + a_{mu_{m}}x^{mk + u_{m}})(b_{00} + b_{01}x + b_{0v_{0}}x^{v_{0}} + b_{10}x^{k} + b_{11}x^{k + 1} + \cdots + b_{1v_{1}}x^{k + v_{1}} + \cdots + b_{n0}x^{nk} + b_{n1}x^{nk + 1} + \cdots + b_{nv_{n}}x^{nk + v_{n}}) = 0$ in $R[x]$. Since $R$ is strongly nil-IFP, so $a_{ic}Rb_{jd} \in N(R),$ for all $0 \leq i \leq m$, $0 \leq j \leq n$, $c \in \{0, 1, \ldots, u_{i}\}$ and $d \in \{0, 1, \ldots, v_{j}\}.$ Therefore, $f_{i}(x)Rg_{j}(x) \in N(R)[x] = N(R[x])$ for all $0 \leq i \leq m$ and $0 \leq j \leq n$. Hence, $R[x]$ is a strongly nil-IFP.
\end{proof}
\begin{cor} Let $R$ be a ring. If $N(R)[x] = N(R[x])$, then the following conditions are equivalent:
\begin{itemize}
\item[$(1)$] $R$ is strongly nil-IFP.
\item[$(2)$] $R[x]$ is strongly nil-IFP.
\item[$(3)$] $R[x, x^{-1}]$ is strongly nil-IFP.
\end{itemize}
\end{cor}
\begin{pro} If $R$ is an Armendariz and $QRPR$, then $R[x]$ is strongly nil-IFP.
\end{pro}
\begin{proof}
By Proposition (\ref{pro5}), $R$ is a strongly nil-IFP. In Armendariz ring, $N(R)[x] = N(R[x])$. It follows by Proposition \ref{pro6}, $R[x]$ is strongly nil-IFP.
\end{proof}

\begin{pro} Let $S$ be a multiplicative closed subset of a ring $R$ consisting of central regular elements. Then $R$ is strongly nil-IFP ring if and only if $S^{-1}R$ is strongly nil-IFP.
\end{pro}
\begin{proof} Let $R$ be a strongly nil-IFP. Let $P(x), Q(x) \in S^{-1}R[x]$ for $P(x) = u^{-1}p(x)$ and $Q(x) = v^{-1}q(x)$ such that $P(x)Q(x) = 0$, where $u, v \in S$. This implies $p(x)q(x) = 0$. Since, $R$ is strongly nil-IFP, so $arb \in N(R)$, for each $a \in C_{p(x)}, b \in C_{q(x)}$ and for all $r \in R$. Therefore, $u^{-1}aS^{-1}Rv^{-1}b \subseteq N(S^{-1}R)$ for each $a \in C_{P(x)}, b \in C_{Q(x)}$. Hence, $S^{-1}R$ is a strongly nil-IFP.
\end{proof}

\begin{pro} Let $R$ be a finite subdirect sum of strongly nil-IFP rings. Then $R$ is strongly nil-IFP.
\end{pro}
\begin{proof} Let $I_{l}$, where $1 \leq l \leq k$ be ideals of $R$ such that $R/I_{l}$ be strongly nil-IFP and $\bigcap_{l=1}^{k}I_{l} = 0$. Suppose $f(x) = \Sigma_{i=0}^{m}a_{i}x^{i}$ and $g(x) = \Sigma_{j=0}^{n}b_{j}x^{j}$ $\in R[x]$ such that $f(x)g(x) = 0$. This implies $\overline{f(x)}\overline{g(x)} = \overline{0}$. Then there exist $s_{ijl} \in N$ such that $(a_{i}Rb_{j})^{s_{ijl}} \subseteq I_{l}$. Set $s_{ij} = s_{ij1}s_{ij2}s_{ij3}\ldots s_{ijk}$, then $(a_{i}Rb_{j})^{s_{ij}} \subseteq I_{l}$ for any $l$. This shows that $(a_{i}rb_{j})^{s_{ij}} = 0$ for each $i, j$ and for all $r \in R.$ Hence, $R$ is strongly nil-IFP.
\end{proof}

\begin{thm} Let $R$ be an algebra over commutative domain $S$ and $D$ be Dorroh extension of $R$ by $S$. Then $R$ is strongly nil-IFP if and only if $D$ is strongly nil-IFP.
\end{thm}
\begin{proof} Let $D$ be a strongly nil-IFP. Since $D$ is a trivial extension of $R$ by $S$. Therefore, $R$ is strongly nil-IFP.\\
 Conversely, let $R$ be strongly nil-IFP. Let $f(x) = \sum_{i=0}^{m}(a_{i}, b_{i})x^{i} = (f_{1}(x), f_{2}(x))$ and $g(x) = \sum(c_{j}, d_{j})x^{j} = (g_{1}(x), g_{2}(x)) \in D[x]$ such that $f(x)g(x) = 0$ where $f_{1}(x) = \sum_{i=0}^{m}a_{i}x^{i},$ $f_{2}(x) = \sum_{i=0}^{m}b_{i}x^{i},$ $g_{1}(x) = \sum_{j=0}^{n}c_{j}x^{j}$ and $g_{2}(x) = \sum_{j=0}^{n}d_{j}x^{j}$. From $f(x)g(x) = 0$, we have

\begin{equation}\label{4}
 f_{1}(x)g_{1}(x)+f_{1}(x)g_{2}(x)+f_{2}(x)g_{1}(x) = 0, f_{2}(x)g_{2}(x) = 0.
\end{equation}
Since $S$ is a domain, therefore either $f_{2}(x) = 0$ or $g_{2}(x) = 0$. \\
\textbf{Case 1.} If $f_{2}(x) = 0$, then from equation (\ref{4}), we have $f_{1}(x)g_{1}(x)+f_{1}(x)g_{2}(x) = 0$ and this implies $f_{1}(x)(g_{1}(x)+g_{2}(x)) = 0$. Since $R$ is strongly nil-IFP, therefore $a_{i}R(c_{j}+d_{j}) \subseteq N(R)$. Hence $(a_{i}, 0)(r, s)(c_{j}, d_{j}) = (a_{i}(r+s)c_{j}+a_{i}(r+s)d_{j}, 0)\in N(R)$ for each $i, j$ and any $(r, s) \in D$. Thus, $D$ strongly nil-IFP.\\
\textbf{Case 2.} If $g_{2}(x) = 0$, then from equation (\ref{4}), we have $(f_{1}(x)+f_{2}(x))g_{1}(x) = 0$. Since $R$ is strongly nil-IFP, therefore $(a_{i}+b_{i})Rc_{j} \subseteq N(R)$ . Hence $(a_{i}, b_{i})(r, s)(c_{j}, 0) = (a_{i}+b_{i}(r+s)c_{j}, 0) \in N(R)$, for each $i, j$ and any $(r, s) \in D$. Thus, $D$ is a strongly nil-IFP.
\end{proof}
Given Example \ref{ex3}, $R$ is a strongly nil IFP but not an abelian. The following example shows that abelian ring  need not be strongly nil IFP.
\begin{ex}(\cite{DDD}, Example (12)) Let $K$ be a field and $A = [a_{0}, a_{1}, b_{0}, b_{1}]$ be the free algebra generated by noncommuting indeterminates $a_{0}, a_{1}, b_{0}, b_{1}$. Let $I$ be an ideal of $A$ generated by $a_{0}b_{0}, a_{0}b_{1}+a_{1}b_{0}, a_{1}b_{1}$ and $R = A/I$. Let $f(x) = a_{0}+a_{1}x, g(x) = b_{0}+b_{1}x$ such that $f(x)g(x) = 0$. But $a_{0}b_{1}b_{1}$ is not nilpotent. Therefore, $R$ is not strongly nil-IFP. By Example 12 of \cite{DDD} $R$ is an abelian ring.
\end{ex}

\section*{Acknowledgement}
The authors are thankful to Indian Institute of Technology Patna for providing financial support and research facilities.


\begin{thebibliography}{00}
\bibitem {D} D. D. Anderson and V. Camillo, Armendariz rings and Gaussian rings,  Comm. Algebra 26(7) (1998), 2265-2272.
\bibitem {RR} R. Antoine, Nilpotent elements and Armendariz rings, J. Algebra {319}(8) (2008), 3128-3140.
\bibitem {E} E. P. Armendariz, A note on extensions of Baer and P.P.-rings, J. Austral. Math. Soc. 18 (1974), 470-473.
\bibitem {H} H. E. Bell, Near-rings in which each element is a power of itself, Bull. Aust. Math. Soc. 2 (1970), 363-368.
\bibitem {G} G. F. Birkenmeier, H. E. Heatherly and E. K. Lee, Completely prime ideals and associated radicals, in: S. K. Jain, S. T. Rizvi (Eds.),
Proc. Biennial Ohio State–Denison Conference (1992), World Scientific, New Jersey (1993), 102-129.
\bibitem {P} P. M. Cohn, Reversible rings, Bull. London Maths. Soc. 31 (1999), 641-648.
\bibitem {C} C. Huh, Y. Lee and A. Smoktunowicz, Armendariz rings and semicommutative rings, Comm. Algebra 30(2) (2002), 751-761.
\bibitem {T} T. W. Hungerford, Algebra, Springer-Verlag, New York, 1974.
\bibitem {DD} D. W. Jung, Y. Lee and H. J. Sung, Reversibility over prime radical, Korean J. Math.  22(2) (2014), 279-288.
\bibitem {DDD} D. W. Jung, N.  K. Kim, Y. Lee and S. P. Yang, Nil-Armendariz rings and upper nilradicals, Internat. J. Algebra Comput. 22(6) (2012), 1-13.
\bibitem {N} N. K. Kim and Y. Lee, Armendariz rings and reduced rings, J. Algebra 223 (2000), 477-488.
\bibitem {NNN} N.  K. Kim, K. H. Lee and Y. Lee, Power series rings satisfying a zero divisor property, Comm. Algebra 34(6) (2006), 2205-2218.
\bibitem {R} R. L. Kruse and D. T. Price, Nilpotent Rings, Gordon and Breach, New York, 1969.
\bibitem {TT} T. K. Kwak, Y. Lee and S. J. Yun, The Armendariz property on ideals, J. Algebra 354(1) (2012), 121-135.
\bibitem {L} L. Liang, L. Wang and L. Zhongkui, On a generalization of semicommutative rings, Taiwanese J. Math. 11(5) (2007), 1359-1368.
\bibitem {Z} Z. Liu and R. Zhao, On weak Armendariz rings, Comm. Algebra {34}(7) (2006), 2607-2616.
\bibitem {GG} G. Marks, On 2-primal ore extensions, Comm. Algebra {29}(5) (2001), 2113-2123.
\bibitem {LL} L. Motais de Narbonne, Anneaux semi-commutatifs et unisériels anneaux dont les idéaux principaux sont idempotents, in:
Proceedings of the 106th National Congress of Learned Societies, Perpignan, 1981, Bib. Nat., Paris, 1982, 71-73.
\bibitem {M} M. B. Rege and S. Chhawchharia, Armendariz rings, Proc. Japan Acad. Ser. A Math. Sci. 73(1) (1997), 14-17.

\end{thebibliography}
\end{document}